\documentclass[11pt]{article}
\usepackage{graphicx, verbatim}
\usepackage[update,prepend]{epstopdf}
\usepackage{amsfonts}
\usepackage{amssymb}
\usepackage{amstext}
\usepackage{amsmath}
\usepackage{amsthm,verbatim}
\usepackage{xspace}
\usepackage{graphicx}
\usepackage{algorithm}

\newtheorem{theorem}{Theorem}
\newtheorem{lemma}[theorem]{Lemma}

\newtheorem{question}[theorem]{Question}
\newtheorem{obs}{Observation}

\newcommand{\E}{\ensuremath{\mathbb E}}
\newcommand{\R}{\ensuremath{\mathbb R}}
\newcommand{\C}{\ensuremath{\mathbb C}}

\newcommand{\F}{\ensuremath{\mathcal F}}

\newcommand{\Vn}{\ensuremath{V_\nu}}
\newcommand{\Vl}{\ensuremath{V_\lambda}}
\newcommand{\Vm}{\ensuremath{V_\mu}}
\newcommand{\clmn}{\ensuremath{c_{\lambda \mu}^\nu}}
\newcommand{\plmn}{\ensuremath{P_{\lambda \mu}^\nu}}
\newcommand{\qlmn}{\ensuremath{Q_{\lambda \mu}^\nu}}
\newcommand{\zlmn}{\ensuremath{\zeta_{\lambda \mu}^\nu}}
\newcommand{\blmn}{\ensuremath{b_{\lambda \mu}^\nu}}
\newcommand{\olmn}{\ensuremath{O_{\lambda \mu}^\nu}}
\newcommand{\Kmd}{\ensuremath{K_{-\de}}}

\newcommand{\lab}{\label}  \newcommand{\ra}{\ensuremath{\rightarrow}}  \def\a{{\mathbf{\alpha}}} \def\de{{\mathbf{\delta}}} \def\De{{{\Delta}}}  
  \def\beq{\begin{eqnarray}} \def\eeq{\end{eqnarray}} \def\ben{\begin{enumerate}}
\def\een{\end{enumerate}} \def\bit{\begin{itemize}}
\def\bel{\begin{lemma}}
\def\eel{\end{lemma}}
\def\eit{\end{itemize}} \def\beqs{\begin{eqnarray*}} \def\eeqs{\end{eqnarray*}} \def\bel{\begin{lemma}} \def\eel{\end{lemma}}
\newcommand{\N}{\mathbb{N}} \newcommand{\Z}{\mathbb{Z}} \newcommand{\Q}{\mathbb{Q}}  
    
    \newcommand{\p}{\mathbb{P}}
    
  \newcommand{\la}{\lambda}  \renewcommand{\a}{\alpha}
\renewcommand{\b}{\beta}
\renewcommand{\t}{\theta}
   \def\eps{{\epsilon}}  \def\ie{i.\,e.\,} 
\def\vol{\mathrm{vol}}

\newcommand{\RR}{\mathbb{R}}

\newcommand{\hf}{{\mathcal{H}}_{f}}

\renewcommand{\u}{u}
\renewcommand{\v}{v}

\newtheorem{thm}{Theorem}[section]

\newtheorem{prop}[thm]{Proposition}
\theoremstyle{definition}
\newtheorem{defn}[thm]{Definition}
\theoremstyle{remark}

\numberwithin{equation}{section}
\begin{document}

\title{Estimating Certain Non-Zero Littlewood-Richardson Coefficients}%
\author{Hariharan Narayanan}
\maketitle
%
%
\begin{abstract}
Littlewood Richardson coefficients are structure constants appearing in the representation theory of the general linear groups ($GL_n$).
The main results of this paper are:
\ben
\item A strongly polynomial randomized approximation scheme  for certain Littlewood-Richardson coefficients.
\item A proof of approximate log-concavity of certain Littlewood-Richardson coefficients.
\een
\end{abstract}

\section{Introduction}
Littlewood Richardson coefficients are structure constants appearing in the representation theory of the general linear groups ($GL_n$). They are ubiquitous in mathematics, appearing in representation theory, algebraic combinatorics, and the study of tilings. They appear in physics in the context of the fine structure of atomic spectra since Wigner \cite{Wigner}. They count the number of tilings using squares and triangles of certain domains \cite{squaretri}. They play a role in Geometric Complexity Theory, which seeks to separate complexity classes such as $P$ and $NP$ by associating group-theoretic varieties  to them, and then proving the non-existence of injective morphisms from one to the other by displaying representation theoretic obstructions \cite{gct1, gct2}.  Thus, computing or estimating Littlewood-Richardson coefficents is important in several areas of science. Results for testing the positivity of a Littlewood-Richardson coefficient may be found in \cite{gct3, BI}. For the case where the Lie group has a fixed rank, efficient (\ie polynomial time) computation is possible based on Barvinok's algorithm (see ``LATTE" \cite{Latte}), or by using vector partition functions (see \cite{billey}). The degree of the polynomial in the runtime depends on the rank. Unfortunately Littlewood-Richardson coefficients are  $\#P-$complete (see \cite{Nar1}), and so in the case of variable rank, under the widely held complexity theoretic belief that $P \neq NP$, they cannot be computed exactly in polynomial time.

We are thus lead to the question of efficient approximation:
\begin{question}
 Is there an algorithm which takes as input the labels $\la, \mu, \nu$ of a Littlewood-Richardson coefficient $c_{\la\mu}^\nu$, and produces in polynomial time an $1\pm \epsilon$ approximation with probability more than $ 1- \de$?
\end{question}
\begin{defn} We say that an algorithm for estimating a quantity $f(x)$, where $x \in \Q^n$ runs in randomized strongly polynomial time, if the number of ``standard" operations that it uses depends polynomially on $n$, but is independent of the bit-length of those rational numbers.
We require that  there be a universal constant $C$ such that the algorithm output a random rational number $\hat{f}(x)$ with the property that $$\p\left[\frac{\hat{f}(x)}{f(x)} \in (1 - C\eps, 1 + C\eps)\right] > 1 - C\de.$$
We allow a polynomial dependence in $n$, $\eps^{-1}$ and the negative logarithm $- \log \de$ but not the bitlength of $x$. Our set of standard operations consists of additions, subtractions, multiplications, divisions, comparisons and taking square-roots. We allow the use of random numbers whose bitlength depends on the bitlength of the input, provided the operations done on them are standard.
\end{defn}

\subsection{Littlewood-Richardson Cone}

Given a symmetric non-negative definite matrix $X$, let $eig(X)$ denote the eigenvalues of $X$ listed in non-increasing order. The Littlewood-Richardson cone (or LRC) is defined as the cone of $3-$tuples $(eig(U), eig(V), eig(W))$, where $U, V, W$ are symmetric positive definite and $U + V = W$.

Knutson and Tao proved the following in \cite{KTW2}.

\begin{theorem}
The Littlewood-Richardson coefficient $\clmn$ is greater than $0$ if and only if $(\la, \mu, \nu)$ is an integer point in the Littlewood-Richardson cone.
\end{theorem}

\begin{figure}\label{fig:LRcone1}
\begin{center}
\includegraphics[height=1.6in]{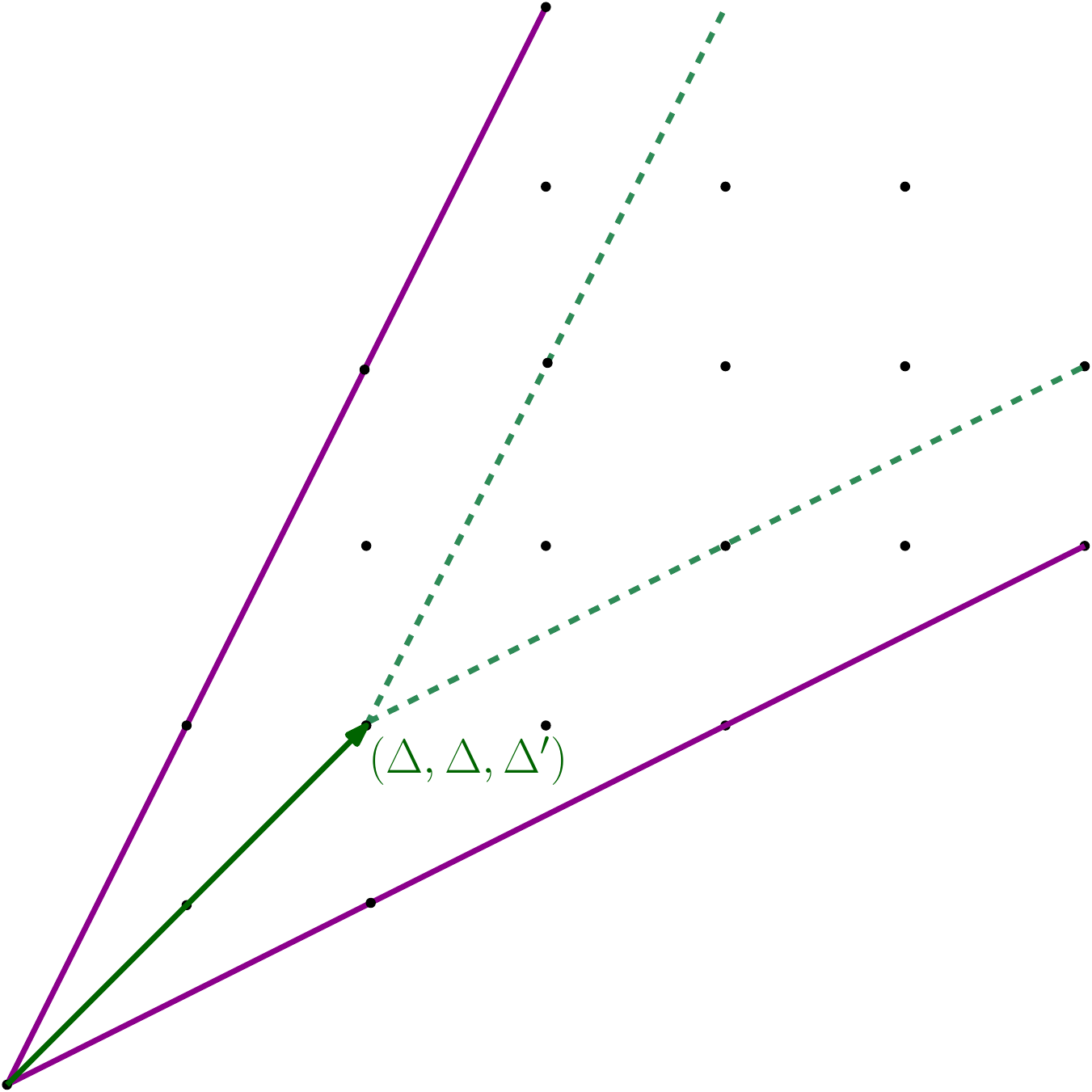}
\caption{$LRC$ and $LRC + (\De, \De, \De')$}
\end{center}
\end{figure}

In the remainder of this paper,  $C, C_1, \dots$ will denote  sufficiently large absolute constants.

In particular, as $\gamma$ tends to infinity, the fraction of all $c_{\la\mu}^\nu$ corresponding to integer points in
$$LRC \cap \left\{\|(\la, \mu, \nu)\|_1  \leq  {\gamma}\right\}$$ that can be approximated, tends to $1$.

Let $\De = 2 (n^3, n^3 - n^2, \dots, n^2)$,

$\De' = ({3n^3}+ {n^2}, {3n^3} - {n^2},  \dots, {n^3} + {3n^2})$.

For $n^2 \eps^{-1} \in \N$, let $\De_\eps = \left(\frac{2}{\eps}\right) (n^3, n^3 - n^2, \dots, n^2)$,
and  $\De_\eps' = \left(\frac{1}{\eps}\right)({3n^3}+ {n^2}, {3n^3} - {n^2},  \dots, {n^3} + {3n^2})$.
 \begin{theorem}\lab{thm:10}
If $(\lambda, \mu, \nu) \in (\De, \De, \De') +  LRC$, then $c_{\la \mu}^\nu$ can be approximated in randomized strongly polynomial time.
\end{theorem}

The following theorem is shown by proving (in Lemma~\ref{lem:11}) that the fraction of all $c_{\la\mu}^\nu$ corresponding to integer points in
$$\left((\De, \De, \De') + LRC\right) \cap \left\{\|(\la, \mu, \nu)\|_1  \leq  {\gamma}\right\}$$ in $LRC \cap \left\{\|(\la, \mu, \nu)\|_1  \leq  {\gamma}\right\}$ is at least $1 - C(\frac{n^5}{\gamma})$.

\begin{theorem}\lab{thm:4}
There is an absolute constant $C$ such that for $n > C$, and $\gamma > C n^5$, there is a randomized strongly polynomial time algorithm for approximating  a $1 - C(\frac{n^5}{\gamma})$ fraction of all $c_{\la\mu}^\nu$ corresponding to integer points in
$$LRC \cap \left\{\|(\la, \mu, \nu)\|_1  \leq  {\gamma}\right\}.$$
\end{theorem}

The following result shows that while Okounkov's question on the log-concavity of Littlewood-Richcardson coeffecients in \cite{Okounkov} has been answered in the negative, a form of approximate log-concavity can be shown to hold among the coefficients corresponding to integer points in 
$(1/\eps)(\De, \De, \De') +  LRC$.

\begin{theorem}\lab{thm:ok}
If $n > C$ and $\theta (\la, \mu, \nu) + (1 - \theta) (\la', \mu', \nu') = (\bar \la, \bar \mu, \bar \nu)$, and each vector indexes a Littlewood-Richardson coefficient and is in $(1/\eps)(\De, \De, \De') +  LRC$, then
$$\log\left(\clmn\right) + C\eps \geq \theta \log\left(c_{\la'\mu'}^{\nu'} \right) +  (1 - \theta) \log\left( c_{\bar{\la}\bar{\mu}}^{\bar{\nu}} \right).
$$
\end{theorem}

\section{Preliminaries}
\subsection{Group Representations}

Suppose $V$ is a complex vector space and $G$ and $GL(V)$ are respectively a group and the group of automorphisms of $V$, then
given a homomorphism $\rho: G \ra GL(V)$, we call $V$ a representation of $V$. If no non-trivial proper subspace of $V$ is mapped to itself by all $g \in G$,  $V$ is said to be irreducible. Littlewood-Richardson coefficients appear in the representation theory of the general linear group
$GL_n(\C)$. Suppose $V_\la$, $V_\mu$ and $V_\nu$ are irreducible representations of $GL_n(\C)$. The Littlewood-Richardson coefficient $c_{\la \mu}^\nu$ is the multiplicity of $\Vn$ in $\Vl \otimes \Vm$.
\beq\Vl \otimes \Vm = \bigoplus_\nu \Vn^{\clmn}.\eeq

\subsection{Hive model and rhombus inequalities for Littlewood-Richardson coefficients}

\begin{figure}\label{fig:tri}
\begin{center}
\includegraphics[scale=0.4]{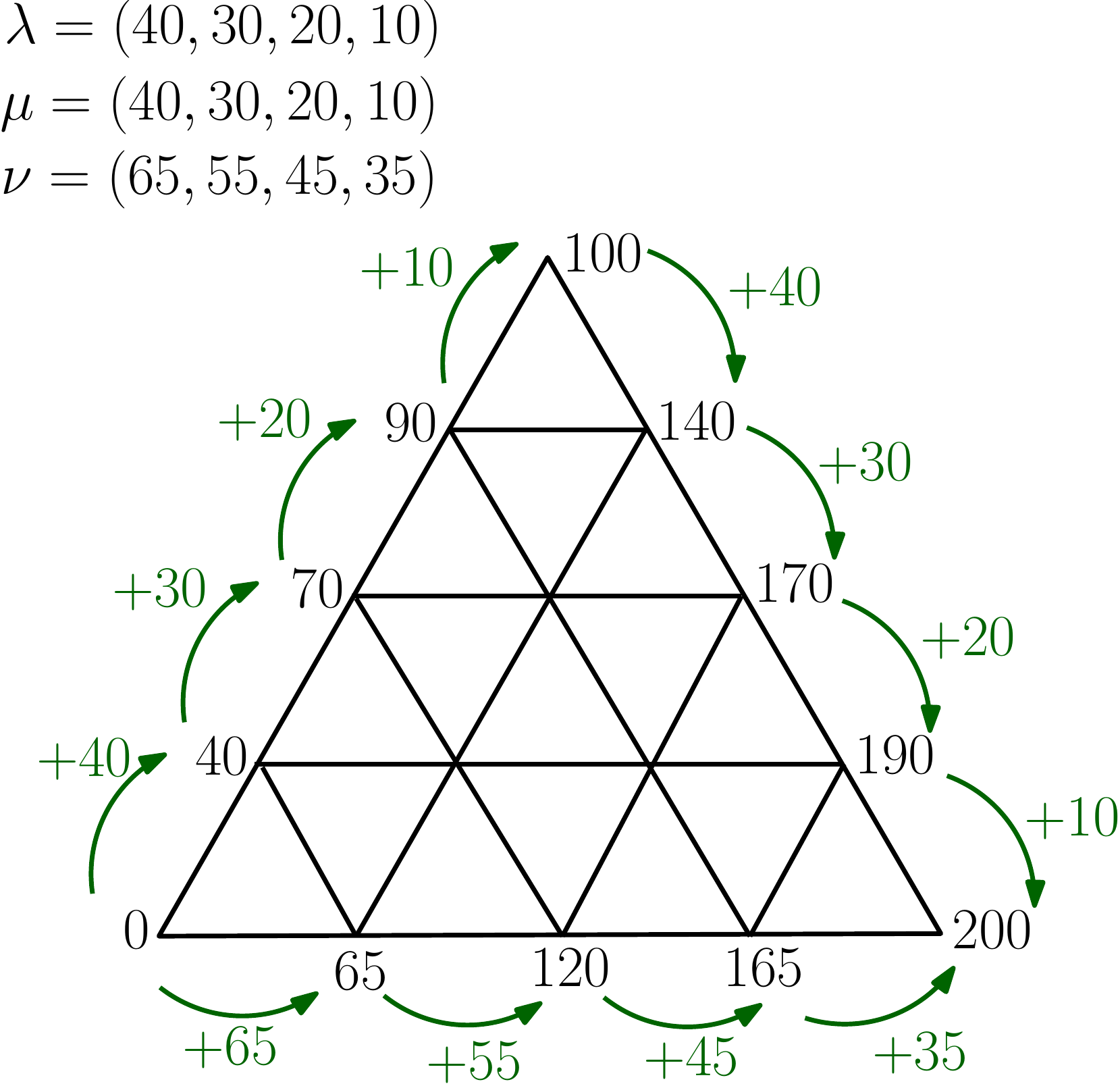}
\caption{Hive model for LR coefficients}
\end{center}
\end{figure}

Let $\la, \mu, \nu$ be vectors in $\Z^n$ whose entries are non-increasing non-negative integers. In all subsequent appearances, this will be assumed of $\la, \mu$ and $\nu$. Let the sum of the entries of a vector $\a$ be denoted $|\a|$.  Further, let $$|\la| + |\mu| = |\nu|.$$ Take an equilateral triangle $\Delta_n$ of side $n$. Tessellate it with unit equilateral triangles. Assign boundary values to $\Delta_n$ as in Figure 1; Clockwise, assign the values $0, \la_1, \la_1 + \la_2, \dots, |\la|, |\la| + \mu_1, \dots, |\la| + |\mu|.$ Then anticlockwise, on the horizontal side, assign  $$0, \nu_1, \nu_1 + \nu_2, \dots, |\nu|.$$

\begin{figure}\label{fig:tri3}
\begin{center}
\includegraphics[scale=0.4]{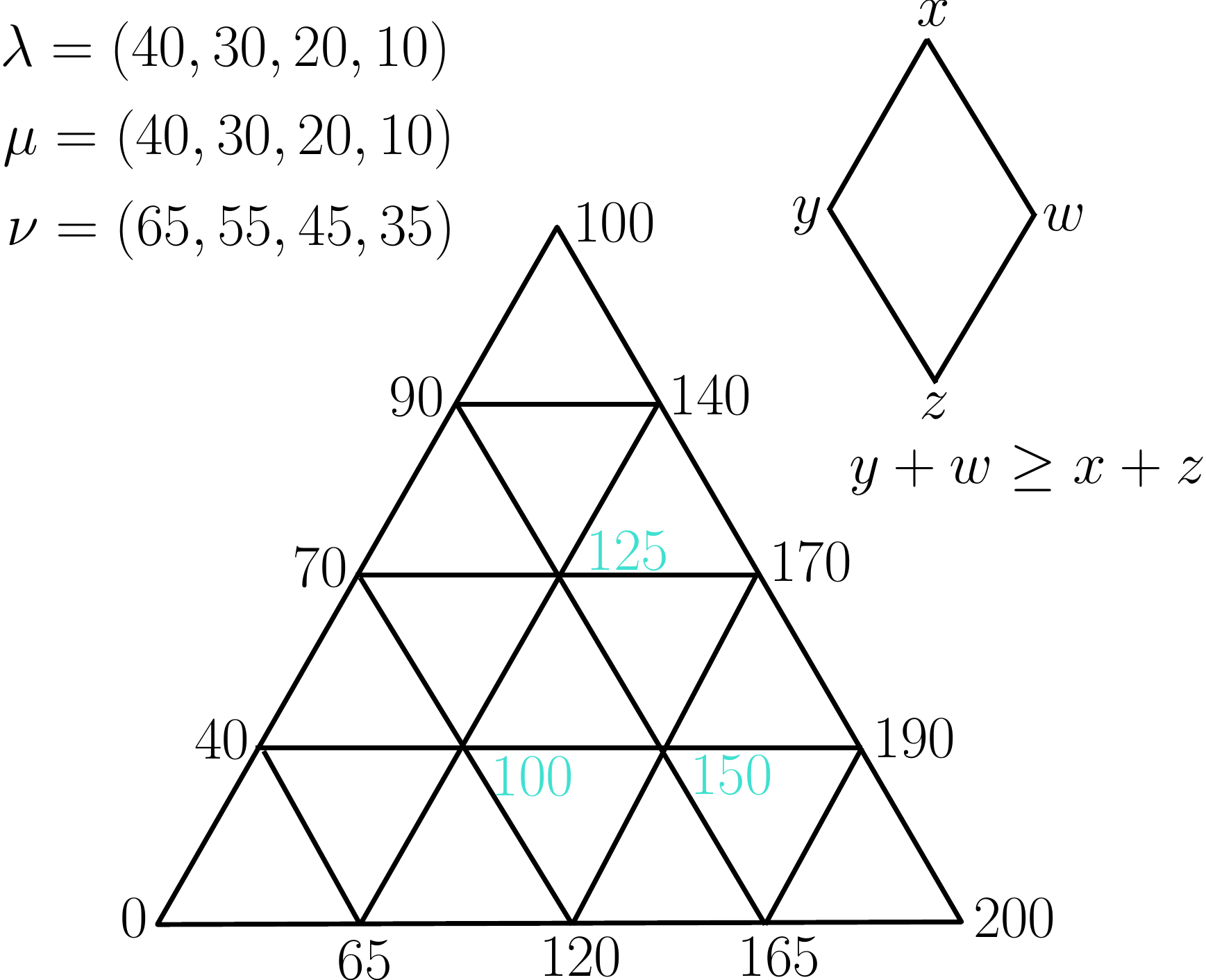}
\caption{Interior nodes in the hive model for LR coefficients}
\end{center}
\end{figure}

Knutson and Tao defined this hive model for Littlewood-Richardson coefficients in \cite{KT1}. They proved that the Littlewood-Richardson coefficient
$\clmn$ is given by the number of ways of assigning integer values to the interior nodes of the triangle, such that the piecewise linear extension to the interior of $\Delta_n$ is a concave function $f$ from $\Delta_n$ to $\R$. Another way of stating this condition is that for every rhombus such as that in Figure 2, if the values taken at the nodes are $w, x, y, z$ where $y$ and $w$ correspond to $120^\circ$ angles, then $y + w \geq x + z$. Let $L$ denote the unit triangular lattice that subdivides $\De_n$. We refer to any map from $\Delta_n \cap L$ to $\Z$ that satisfies the rhombus inequalities as a {\it hive} .

\begin{defn}
 Let $\plmn$ be a subset of $\R^{n \choose 2}$. Let the canonical basis correspond to the set of interior nodes in the corresponding hive. (see Figure 2).
 Let $\plmn$ denote the hive polytope corresponding to $(\la, \mu, \nu)$ defined by the above inequalities, one for every unit rhombus. 
\end{defn}

\section{A Randomized Approximation Scheme}

\begin{center}
\includegraphics[scale=0.3]{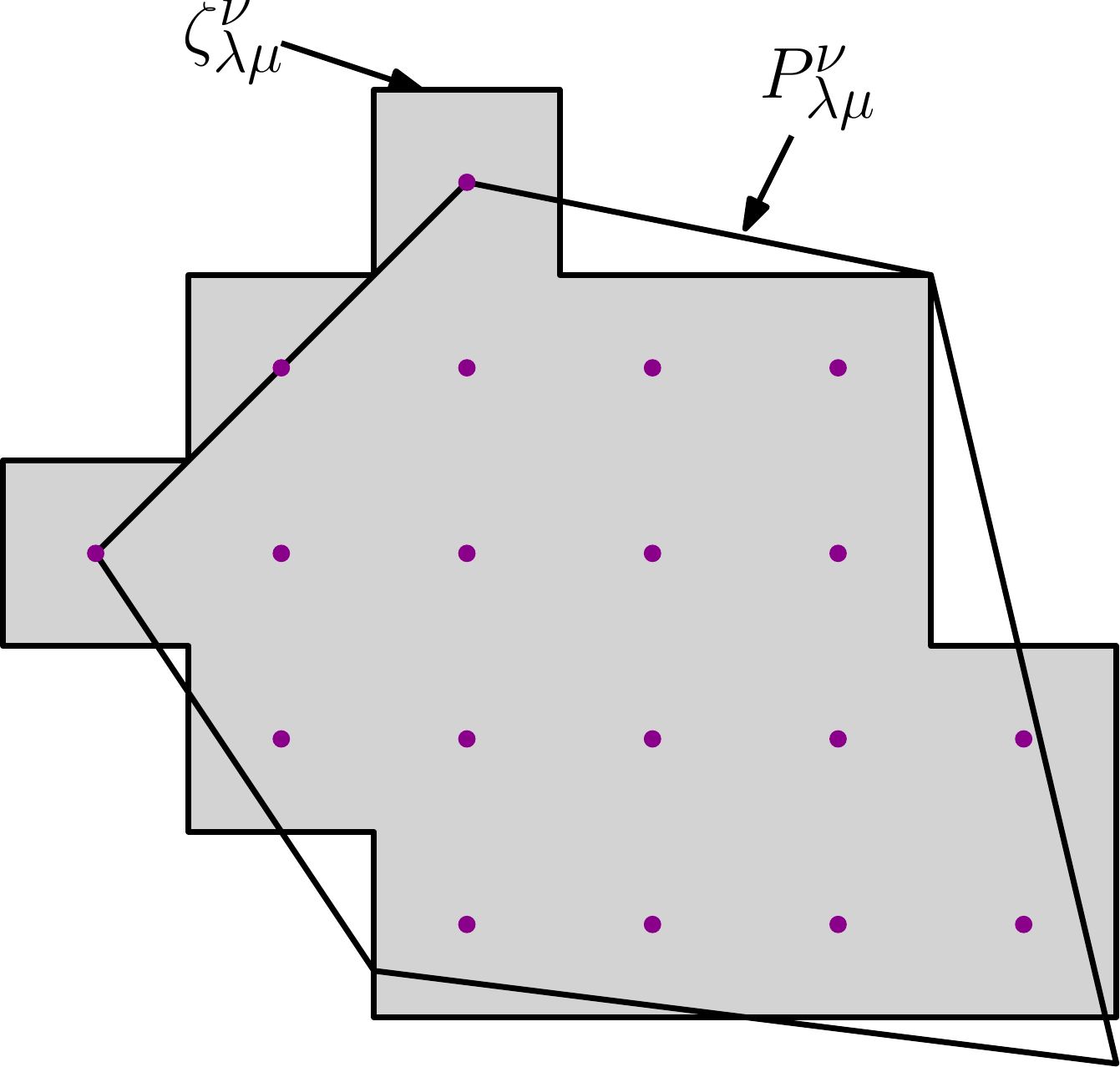}
\end{center}


The number of points in $\plmn \cap \Z^{n\choose 2}$ is equal to the volume of the set  \beq  \zeta_{\la\mu}^\nu := \left\{x \big | \inf\limits_{y\in P_{\la\mu}^\nu \cap \Z^{{n \choose 2}}}\| x - y\|_\infty < \frac{1}{2}\right\}.\eeq

Let $\plmn$ be described as
\beq\lab{eq:blmn} Ax - \blmn \preccurlyeq \vec{0},\eeq where the rows of $A$ correspond to rhombus inequalities.

We make the following observation.
\begin{obs}\lab{obs:1}
 Each row of $A$ has at most $4$ non-zero entries. Each of these entries equals $\pm 1$.
\end{obs}
 The vector $\blmn$  depends on $\la, \mu, \nu$ in a way that reflects the boundary conditions of the hive.

Let $\qlmn$ denote the polytope defined by the inequalities
\beq\lab{eq:qlmn} Ax - \blmn \preccurlyeq \vec{2}.\eeq

Let $\olmn$ denote the polytope defined by the inequalities
\beq\lab{eq:olmn} Ax - \blmn \preccurlyeq -\vec{ 2}.\eeq

By (\ref{eq:blmn}), (\ref{eq:qlmn}) and (\ref{eq:olmn}), \beq\lab{eq:z} \olmn \subseteq \zeta_{\la\mu}^\nu \subseteq \qlmn.\eeq

  In what follows, we were strongly influenced by a result of Kannan and Vempala, who  show in \cite{KannanVempala} that the number of integer points in an $n-$dimensional polytope with $r-$faces containing  a Euclidean ball of radius $O(n\sqrt{\log r})$ is within a constant factor of the volume. However, by exploiting Observation~\ref{obs:1}, we avoid the multiplicative factor of  $\sqrt{\log r}$.

\begin{lemma}\lab{lem:6}
If $(\lambda, \mu, \nu) \in (\De_\eps, \De_\eps, \De_\eps') +  LRC$, then $\plmn$ contains a Euclidean ball of radius $\frac{{n \choose 2}}{2\eps}$.
\end{lemma}
\begin{proof}
By Results of Knutson-Tao \cite{KT1} and Knutson-Tao-Woodward \cite{KTW2}, for every integer point $({\la}', \mu', \nu')$ in LRC,  $c_{{\la}' \mu'}^{\nu'} > 0$.  Fix $\la' = \la - {\De_\eps}$, $\mu' = \mu - {\De_\eps}$ and $\nu' = \nu - {\De'_\eps}$. Let  $z$ denote an arbitrary integer point in    $P_{{\la}'\mu'}^{\nu'}$ (which exists by the last remark).
By (\ref{eq:blmn}) and the corresponding set of inequalities for $P_{\De_\eps\De_\eps}^{\De'_\eps}$, we see that $z + P_{\De_\eps\De_\eps}^{\De'_\eps} \subseteq \plmn$. To prove the lemma, it thus suffices to show that
$P_{\De_\eps\De_\eps}^{\De'_\eps}$ contains a ball of radius ${n \choose 2}$. We first describe the center of this ball. Let $(0, 0)$ be the leftmost -- bottommost corner of the hive in question.
Here $u$ is the coordinate corresponding to the $x$ direction, and $v$ is the coordinate corresponding to the $x/2 + \frac{\sqrt{3}}{2}y$ direction.
Consider the restriction of the following function to $\De_n \cap L$:
$$f(u, v) = \left(\frac{1}{\eps}\right)\left((3n^3 +  n^2) u + (2 n^3 ) v - (n^2 - n)(u^2 + v^2)\right).$$ This corresponds to a hive. Moreover, the resulting vector $x_f$ satisfies
\beq\lab{eq:rlmn} Ax_f - b_{\De_\eps\De_\eps}^{\De'_\eps} \preccurlyeq - \left(\frac{2}{\eps}\right)\vec{{n\choose 2}}.\eeq
Therefore, by Observation~\ref{obs:1}, for any vector $y \in \R^{{n \choose 2}}$  such that $\|y\|_\infty \leq  \frac{{n \choose 2}}{2\eps}$.
\beq\lab{eq:rlmn} A(x_f + y) - b_{\De_\eps\De_\eps}^{\De'_\eps} \preccurlyeq \vec{0}.\eeq In particular,
this means that $\plmn$ contains a ball of radius $\frac{{n \choose 2}}{2\eps}$.
\end{proof}

\begin{lemma}\lab{lem:z} If $n > C$, and $n^2 \eps^{-1} \in \N$ then if $(\lambda, \mu, \nu) \in (\De_\eps, \De_\eps, \De'_\eps) +  LRC$
 \beq  1 - {C}{\eps} \leq \frac{\vol\, \zlmn}{\vol\, \qlmn} \leq  1.\eeq
 \end{lemma}

\begin{proof}
Let $\la' = \la - {\De_\eps}$, $\mu' = \mu - {\De_\eps}$ and $\nu' = \nu - {\De'_\eps}$. By Results of Knutson-Tao \cite{KT1} and Knutson-Tao-Woodward \cite{KTW2}, $P_{\la'\mu'}^{\nu'}$ contains an integer point. Let this point be $z$.
\begin{lemma}
Let the origin be translated so that centered at $z + (\De_\eps, \De_\eps, \De'_\eps)$ is the new origin. Let $d:= {n \choose 2}$ and  $n \geq C$. Then,
\beq \left(1 + \frac{C_1\eps }{d}\right)\olmn \supseteq \qlmn. \eeq
\end{lemma}
\begin{proof}
This follows from Lemma~\ref{lem:6} and elementary geometry.
\end{proof}
Consequently
\beq \vol\, \zlmn \geq \left(1 + \frac{C_1\eps}{d}\right)^{-d}\left(\vol\, \qlmn\right).\eeq

Thus,
\begin{eqnarray} \vol\, \zlmn & \geq & \left(1 + \frac{C_1\eps}{d}\right)^{-d}\left(\vol\, \qlmn\right).\\
& \geq & e^{-C_1\eps}\left(\vol\,\qlmn\right). \end{eqnarray}

\end{proof}
The following theorem is a special case of Kirszbraun's Theorem \cite{Kirszbraun}.
\begin{theorem}\lab{thm:kz}
 if U is a subset of $\R^2$,  and

    $$f : U \ra \R$$

is a Lipschitz-continuous map, then there is a Lipschitz-continuous map

    $$F: \R^2 \ra \R$$

that extends $f$ and has the same Lipschitz constant as $f.$

\end{theorem}

\begin{lemma}\lab{lem:cor:1}
Given $(\a, \b, \theta)$ where each each vector is in $ \R^n$ and $$\sum_i \a_i + \b_i = \sum_i \theta_i,$$ suppose that $4\|(\a, \b, \theta)\|_\infty \leq \de.$ Assign boundary values to $\Delta_n$ as in Figure 1; Clockwise, assign the values $0, \a_1, \a_1 + \a_2, \dots, |\a|, |\a| + \b_1, \dots, |\a| + |\b|.$ Then anticlockwise, on the horizontal side, assign  $$0, \theta_1, \theta_1 + \theta_2, \dots, |\theta|.$$
Then there exists a assignment of real values to the interior nodes of $\De_n$ that has the property that the Lipschitz constant of the resulting map from $\De_n \cap L$ to $\R$ is less or equal to $\de/2$. Consequently no rhombus inequality of the form $w + y - x - z \geq 0$ (see Figure 2) is violated by more than $\de$.
\end{lemma}
\begin{proof} We apply Theorem~\ref{thm:kz} after setting $U$ to be the the set of boundary vertices of $\De_n$. In order to get an upper bound on the Lipschitz onstant of $f$ corresponding to boundary data $\a, \b, \t$, it suffices to consider the worst-case geometrical configuration involving a 60 degree angle. More precisely, we can bound the Lipschitz constant by
\beq \sup_{ABC}\left(\de/4\right) \frac{|AB| + |BC|}{|AC|}.\eeq
as $ABC$ ranges over all triangles having a 60 degree angle at $B$. By the cosine law, \beq |AB|^2 + |BC|^2 - |AB|\cdot|BC| = |AC|^2.\eeq
Therefore, using the A.M-G.M inequality $$(|AB| + |BC|)^2 -  \frac{3}{4} (|AB| + |BC|)^2 \leq |AC|^2.$$ This tells us that
$$\frac{|AB| + |BC|}{|AC|} \leq 2.$$
Therefore, The Lipschitz constant of $f$ is at most \beq \sup_{ABC}\left(\de/4\right) \frac{|AB| + |BC|}{|AC|} \leq \de/2.\eeq  We apply Theorem~\ref{thm:kz} and note that $w + y - x - z = (w - x) + (y - z)$ and so the maximum value  of $w + y - x - z$ (see Figure 2) taken around a unit rhombus is twice the Lipschitz constant. This proves the Lemma.
\end{proof}
\begin{defn}
We denote the set of points within a distance $\delta$ of a convex body $K$
(including $K$ itself) by $K_\de$. This is called the
{\it outer parallel body} of $K$ and is convex. The set of points
at a distance $\geq \delta$ to $\R^n\setminus K$ shall be denoted $K_{-\de}$.
This is called the {\it inner parallel body} of $K$ and is
convex as well.
\end{defn}

\begin{figure}\label{fig:cone}
\begin{center}
\includegraphics[height=2in]{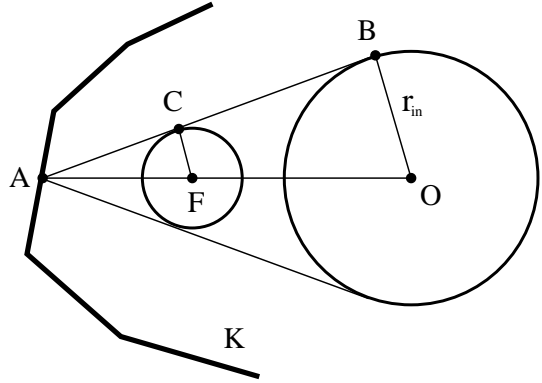}
\end{center}
\caption{$\Kmd$ contains $\left(1 - \frac{\delta}{r_{in}}\right)K$}
\end{figure}

The following Lemma appeared in \cite{heatflow}.
\begin{lemma}\label{lemma:1} Let $K$ contain a ball of radius $r_{in}$ centered at the origin. Then, 
$$
 \Kmd \supseteq \left(1 - \frac{\delta}{r_{in}} \right) K
$$
\end{lemma}
\begin{proof}
Let $O$ be the center of the sphere of radius $r_{in}$ contained inside
$K$.   Let $A$ be a point on $\partial K$ and let $F :=  \left(1 - \frac{\delta}{r_{in}}\right) A$.  It suffices to prove that $F \in \Kmd$.

We construct the smallest cone from $A$ containing the sphere. Let
$B$ be a point where the cone touches the sphere of radius $r_{in}$ centered at the origin. We have $OB =
r_{in}$. Now consider the inscribed sphere centered at $F$. By
similarity of triangles, we have
$$
\frac{CF}{OB} = \frac{AF}{AO}.
$$
Noticing that $AF = \frac{\delta}{r_{in}}\, OA$, we obtain
$$
CF = OB \frac{AF}{AO} =  \delta.
$$
We thus see that the radius of the inscribed ball is $\delta$ and
hence the $\delta$-ball centered in $F$ is contained in $K$. Therefore, $F \in \Kmd$.

\end{proof}
\begin{lemma}\lab{lem:11}
Let $\bar\gamma := \frac{n^5}{\gamma}$. There is an absolute constant $C$ such that if $n > C$ and $0 < \bar\gamma < C^{-1}$, then the fraction of integer points \beq(\la, \mu, \nu) \in LRC \cap \{(\la, \mu, \nu)\big|\|(\la, \mu, \nu)\|_1 \leq \frac{C n^5}{\bar\gamma}\}\eeq such that $$(\la, \mu, \nu) \in (\De, \De, \De') + LRC$$ is greater or equal to $1 - C{\bar\gamma})$.
\end{lemma}
\begin{proof}
Let $D:= LRC \cap \{\|(\la, \mu, \nu)\|_1 \leq \frac{C n^5}{\bar\gamma}\}$.
For $(\la, \mu, \nu) \in LRC$, $\|(\la, \mu, \nu)\|_1 = \sum_{i=1}^n (\la_1 + \mu_1 + \nu_i)$. As a result, the  integer points inside \beq  \left(LRC + (\De, \De, \De')\right) \cap \{\|(\la, \mu, \nu)\|_1 \leq \frac{Cn^5}{\bar\gamma}\}\eeq
are in bijective correspondence (via subtraction of $(\De, \De, \De')$) with the integer points in
\beq D' :=  LRC  \cap \{(\la, \mu, \nu) \big | \|(\la, \mu, \nu)\|_1 \leq \frac{Cn^5}{\bar\gamma} - \|(\De, \De, \De')\|_1\}.\eeq
We will first show that $D'$ contains a ball whose radius is large. Suppose $(\la^o, \mu^o, \nu^o) = \left(\frac{n^2}{8\bar\gamma}\right)(\De, \De, \De')$. Suppose $4\|\a, \b, \t\|_\infty < \left(\frac{ (n^3 )}{16\bar\gamma}\right) {n \choose 2}.$ By Lemma~\ref{lem:cor:1} $$(\la^o + \a, \mu^o + \b, \nu^o + \t) \in D'.$$ Thus $D'$ contains a cube of side  $2R:=2\left(\frac{ n^2 }{16\bar\gamma}\right) {n \choose 2}$, and a hence a Euclidean ball of radius $R$ centered at $(\la^o, \mu^o, \nu^o)$. Let $D'_{-\sqrt{3n}}$ be the  inner parallel body of $D'$ at a distance of $\sqrt{3n}$.  Let $D_{\sqrt{3n}}$ be the  outer parallel body of $D'$ at a distance of $\sqrt{3n}$.  Then, we have the following lemma.
\begin{lemma}
Translate the origin to $(\la^o, \mu^o, \nu^o)$.  There is an absolute constant $C_1$, such that for $n \geq C_1$,  \beq \left(1 + \frac{\bar\gamma}{ n}\right) D'_{-\sqrt{3n}} \supseteq D_{\sqrt{3n}}.\eeq
\end{lemma}
\begin{proof}
Let $K := D_{\sqrt{3n}}$ and $\de := {\sqrt{3n}}$. Then, $K_{-\de} = D$, and applying Lemma~\ref{lemma:1}, we have  \beq \lab{eq:lll:1}D_{\sqrt{3n}} \subseteq \left(1 - \frac{\sqrt{3n}}{R + \sqrt{3n}} \right)^{-1}D.\eeq
 Next, because $D'$ and $D$ are homothetic cones (with respect to the common apex), we have
 \beq\lab{eq:lll:2} \left(1 - \frac{\bar\gamma\|(\De, \De, \De')\|_1}{Cn^5}\right) D \subseteq D'.\eeq
 Lastly, using Lemma ~\ref{lemma:1}, we have  \beq \lab{eq:lll:3}  \left(1 - \frac{\sqrt{3n}}{R}\right) D' \subseteq D'_{-\sqrt{3n}}.\eeq

The lemma follows from (\ref{eq:lll:1}), (\ref{eq:lll:2}) and (\ref{eq:lll:3}).
 \end{proof}
Therefore
\beq  \vol\, D'_{-\sqrt{3n}} & \geq &  e^{-{C}{\bar\gamma}} \left(1 + \frac{C\bar\gamma}{n}\right)^n \vol\, D'_{-\sqrt{3n}}\\ & \geq &  (1 - {C}{\bar\gamma}) \vol\, D_{\sqrt{3n}}.\eeq
  The volume $\vol\, D'_{-\sqrt{3n}}$ is less or equal to the number of lattice points in $D'$, since the union of all unit cubes centered at lattice points in $D'$ contains $ D'_{-\sqrt{3n}}$. Also, $\vol\, D_{\sqrt{3n}}$ is greater or equal to the number of lattice points in $D$, since the union of all unit cubes centered at lattice points in $D$ is contained in $ D_{\sqrt{3n}}$. This proves the lemma.
\end{proof}


 \begin{center}
\includegraphics[scale=0.6]{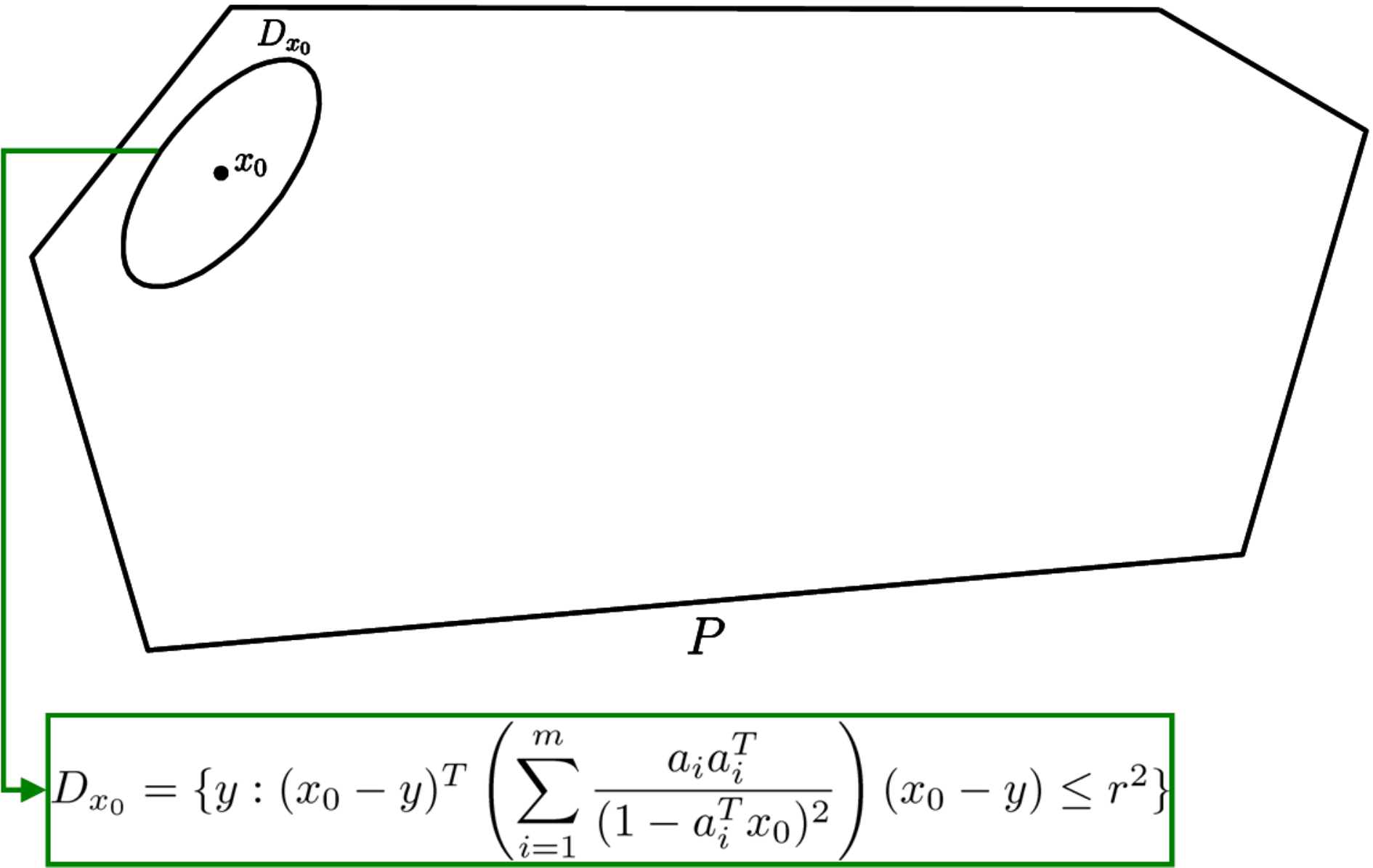}
\end{center}

%
%
%
%
%
\begin{proof}[Proof of Theorem~\ref{thm:10}]
We denote by $W_\infty(\nu_1, \nu_2)$ the infinity-Wasserstein distance between two measures $\nu_1$ and $\nu_2$ supported on a metric space $X$. This is defined as
\beq \inf_\nu \sup\limits_{(x_1, x_2) \in supp(\nu)}\|x_1-x_2\|, \eeq where the infimum is over all measures on $X \times X$ such that the marginal on the first factor is $\nu_1$ and the marginal on the second factor is $\nu_2$.
The Dikin walk \cite{KNsample}, started at the point $x_0 \in \qlmn$ can be used to sample from a distribution $\mu$ that satisfies the following property in strongly polynomial time:
there is a measure $\mu'$ such that $W_\infty(\mu',\mu) < e^{-(\|(\la, \mu, \nu)\|_1)\left(\frac{n}{\eps}\right)^C}$ and $\|\mu' - \mu''\|_{TV} < \eps$, where $\mu''$ is the uniform measure on $\qlmn$. We note that this involves using random numbers whose bitlength depends on $\|(\la, \mu, \nu)\|$, however the number of operations on these random numbers depends polynomially on $n$, $\eps$ and $\log \frac{1}{\de}$.

This us allows produce $s$ i.i.d random points in  $T\qlmn$, each from a distribution $\mu$ that is close to $\mu''$ in the above sense.


\subsection{Algorithm}
\begin{algorithm}

\ben
\item Produce an estimate $\hat{V}$ of the volume $V$ of the the polytope  $Q_{\la\mu}^\nu$ in strongly polynomial time that has the following property:
    \beq \p\left[\frac{\hat{V}}{V} \in [1 - \eps, 1 + \eps]\right] > 1 - \delta.                \eeq
\item Produce $s = \frac{C \log \frac{1}{\de}}{\eps^2}$ from a distribution $\mu$ such that the following holds. There is a measure $\mu'$ such that $W_\infty(\mu',\mu) < e^{-(\|(\la, \mu, \nu)\|_1)\left(\frac{n}{\eps}\right)^C}$ and $\|\mu' - \mu''\|_{TV} < \eps$, where $\mu''$ is the uniform measure on $\qlmn$. Take the nearest lattice point to each sample, and compute the proportion $f$ of the resulting points that lie in $P_{\la\mu}^\nu$.
\item Output $f \hat{V} $  (an estimate of $\vol \,\zeta_{\la\mu}^\nu$).
\een

\end{algorithm}

\subsection{Step 1.}

There are a number of algorithms that compute estimates of the volume of a convex set $K$ in polynomial time as a function of $n$ and the ratio between a the radius $R_K$ of a ball containing $K$ and the radius $r_K$ of another ball contained in $K$. The most efficient of these is the algorithm of Lov\'{a}sz and Vempala \cite{volume4}. However, we wish to compute this estimate using a number of operations that is polynomial in $n$ rather than the bitlength, and therefore we cannot afford any dependence on $\frac{R_K}{r_K}$. To this end,  given  polytope $\qlmn$, we describe below, how to find in strongly polynomial time, a linear transformation $T$ such that for a $T\qlmn$ contains a ball of radius $1$ and is contained inside a ball of radius $m^\frac{3}{2}$, where $m$ is the number of constraints.

 Let $Q$ be a polytope given by $Ax \preccurlyeq 1$. Then, the Dikin ellipsoid $D_{x_0}(r)$ of radius $r$ centered at a point $x_0$ is the
 set of all $y$ such that \beq \{y:(x_0 - y)^T\left(\sum_{i=1}^m \frac{a_i a_i^T}{(1 - a_i^T x_0)^2}\right)(x_0 - y) \leq r^2\}. \eeq

 For every codimension $1$ facet $f$ of $\qlmn$, consider the vector $v_f$ orthogonal to the hyperplane containing $f$. Use the strongly polynomial time linear programming algorithm of Tard\"{o}s as in \cite{gct3} to maximize both $\langle x, v_f \rangle$ and $\langle x, - v_f \rangle$ over all $x$ in $\qlmn$. If both of the resulting points are contained in $f$, we declare the polytope $\qlmn$ to be contained in the affine span of $f$, and therefore have $0$ volume. Otherwise, exactly one of the points is not in $f$. Denote this point by $x_f$. Let $x_0$ be the average of all points $x_f$ as $f$ ranges over the codimension $1$ facets of $\qlmn$. Define $T$ to be a linear transformation that maps the Dikin ellipsoid $D_{x_0}(1)$ of $\qlmn$ onto the unit ball. Suppose $\qlmn$ is expressed as
 $Cx \preccurlyeq 1$.
 Such a $T$ can be found using the Cholesky decomposition of
 $$\left(\sum_{i=1}^m \frac{c_i c_i^T}{(1 - c_i^T x_0)^2}\right),$$ where $m$ is the number of rows in $C$,
 which can be found in strongly polynomial time provided one can find square-roots in one operation. We have assumed this in our model of computation.
\subsubsection{Correctness of Step 1.}

Let $K \cap (x-K)$ be defined to be the symmetrization around $x$ of the convex set $K$. Then, translating the origin to $x_0$ we have the following lemma.
\begin{lemma}
\ben\item $\frac{1}{\sqrt{m}} (\qlmn \cap (- \qlmn)) \subseteq  D_{x_0}(1) \subseteq \qlmn \cap ( - \qlmn).$
\item $ \frac{1}{m-1}\qlmn \subseteq \qlmn \cap (- \qlmn).$
\een
\end{lemma}
\begin{proof}
Consider an arbitrary chord of $ab$ of $\qlmn$ through the origin $x_0$. Identify it with the real line. Let $\pm t$ be the points where the chord intersects the ellipsoid. Let $\pm t_1, \pm t_2, \dots$ be the intersections of the chord with the extended facets of the symmetrized body. Then, $$\frac{1}{t^2} = \sum_i \frac{1}{t_i^2}.$$ It follows that
\beq \frac{1}{\sqrt m} \min_i |t_i| \leq |t| \leq \min_i |t_i|. \eeq This completes the proof of the first part of the lemma. To see the second part, consider again the same arbitrary chord $ab$ through $x_0$. Suppose without loss of generality that $|t_1| \leq |t_2| \leq \dots$, and that $a = - t_1$ and $b = t_k$. Let $a$ lie on the face $f$. Let $a'$ be the intersection with the hyperplane containing face $f_1 := f$ of the line through $x_0$ and $x_f$. Then, by the definition of $x_f$, \beq \frac{|x_f|}{|a'|} \geq \frac{|b|}{|a|}.\eeq
It thus suffices to show that $ m -1 \geq \frac{|x_f|}{|a'|}$. As $f_i$ ranges over the faces of $\qlmn$,  all the points $x_{f_i}$ lie on the same side of the affine span of $f_1$. Therefore their average $x_{2, av}:= \left(\frac{1}{m-1}\right)\sum_{i \geq 2} x_{f_i}$ lies on the same side of the affine span of
$f_1$ as does $x_{f_1}$. However $x_{2, av}$ lies on the line joining $x_{f_1}$ and $x_0 = \left(1/{m}\right)\sum_{i \geq 1} x_{f_i}.$ Therefore \beq m-1 = \frac{|x_f|}{|x_{2, av}|}\geq \frac{|x_f|}{|a'|} \geq \frac{|b|}{|a|}.\eeq This proves the lemma.
\end{proof}

\subsection{Step 2.}

We observe that $f$ is the average of i.i.d $0-1$ random variables $x_i$, each having a success probability $p$ that satisfies  $p - \eps <   \frac{\vol\,\zlmn}{\vol\,\qlmn} < p + \eps.$ By (\ref{lem:z}) $\frac{1}{C} \leq p \leq 1$ for each $i$.

The following inequality is a consequence of Theorem 1 of Hoeffding \cite{Hoeffding}.
\begin{prop}
Let a coin have success probability $p$. Let $\hat{m}$ be the number of successes after $m$ proper trials. Then
$$\p\left[|\frac{\hat m}{m}- p|\geq\lambda p\right] \leq 2e^{-\frac{\la^2 m p }{3}}.$$
\end{prop}

By the above proposition,
\beq \p\left[ \bigg|f - \frac{\vol\,\zlmn}{\vol\,\qlmn} \bigg| \geq 2 \eps \right] \leq \de. \eeq

\subsubsection{Correctness of Step 2.}

Provided we work with numbers whose bitlength is $$(\|(\la, \mu, \nu)\|_1)\left(\frac{n}{\eps}\right)^C,$$ we can produce $\mu$ such that there exists $\mu'$ satisfying the following. \beq\lab{cond:1}W_\infty(\mu',\mu) < e^{-(\|(\la, \mu, \nu)\|_1)\left(\frac{n}{\eps}\right)^C}\eeq and \beq\lab{cond:2}\|\mu' - \mu''\|_{TV} < \eps,\eeq where $\mu''$ is the uniform measure on $\qlmn$. In a real number model of computation,  Dikin walk produces a sample from $\mu'$ in polynomial time. By truncating these real numbers at every step to a bit-length of $(\|(\la, \mu, \nu)\|_1)\left(\frac{n}{\eps}\right)^C,$ the errors do not accumulate to beyond a multiplicative factor of $O(n^C)$ and the resulting measure $\mu$ satisfies the above conditions (\ref{cond:1}) and (\ref{cond:2}).

This completes the proof of Theorem~\ref{thm:10}.
\end{proof}
\begin{proof}[Proof of Theorem~\ref{thm:4}]
This follows from Theorem~\ref{thm:10} and Lemma~\ref{lem:11}.
\end{proof}

\section{Approximate Log-Concavity}

In \cite{Okounkov}, Okounkov raised the question of whether the Littlewood-Richardson coefficients $\clmn$ are a log-concave function of $\la, \mu, \nu$. Chindris, Derkson and Weyman showed in \cite{chindris} that this is false in general. In this section, we show that certain Littlewood-Richardson coefficients satisfy  a form of approximate log-concavity by proving Theorem~\ref{thm:ok}.

\begin{theorem}[Brunn-Minkowski]
Let $Q_1, Q_2$ and $Q_3$ be convex subsets of $\R^n$ and $\theta \in (0, 1)$, such that $\theta Q_1 + (1 - \theta) Q_3 \subseteq Q_2.$ where $+$ denotes Minkowski addition.
Then, $$\left(\vol Q_2\right)^{\frac{1}{n}} \geq \theta \left(\vol Q_1\right)^{\frac{1}{n}} +  (1 - \theta) \left(\vol Q_3\right)^{\frac{1}{n}}.$$
\end{theorem}

Recall that (\ref{eq:qlmn}), $\qlmn$ is described by the system of inequalities
\beqs  Ax - \blmn \preccurlyeq \vec{2}.\eeqs

If $\theta (\la, \mu, \nu) + (1 - \theta) (\la', \mu', \nu') = (\bar \la, \bar \mu, \bar \nu)$, and each vector indexes a Littlewood-Richardson coefficient, then, because $\qlmn, Q_{\la'\mu'}^{\nu'}$ and $Q_{\bar{\la}\bar{\mu}}^{\bar{\nu}}$ are described by  rhombus inequalities, we have
\beq \theta \qlmn + (1 - \theta)Q_{\la'\mu'}^{\nu'} \subseteq Q_{\bar{\la}\bar{\mu}}^{\bar{\nu}}.\eeq
Therefore, by the Brunn-Minkowski inequality,
$$\left(\vol \qlmn\right)^{\frac{1}{n}} \geq \theta \left(\vol Q_{\la'\mu'}^{\nu'} \right)^{\frac{1}{n}} +  (1 - \theta) \left(\vol Q_{\bar{\la}\bar{\mu}}^{\bar{\nu}} \right)^{\frac{1}{n}}.$$ Hence, by the concavity and monotonicity of the logarithm,
$$\log\left(\vol \qlmn\right) \geq \theta \log\left(\vol Q_{\la'\mu'}^{\nu'} \right) +  (1 - \theta) \log\left(\vol Q_{\bar{\la}\bar{\mu}}^{\bar{\nu}} \right).
$$
By Lemma~\ref{lem:z},
if $n > C$, and $n^2 \eps^{-1} \in \N$ then if $(\lambda, \mu, \nu) \in (1/\eps)(\De, \De, \De') +  LRC$
 \beqs  1 - {C}{\eps} \leq \frac{\clmn}{\vol\, \qlmn} \leq  1,\eeqs
and corresponding statements hold for $(\la', \mu', \nu')$ and $(\bar \la, \bar \mu, \bar \nu)$.

This proves Theorem~\ref{thm:ok}.

\section{Concluding Remarks}

In this paper, we developed a strongly polynomial randomized approximation scheme for Littlewood-Richardson coefficients that belong to a translate of the Littlewood-Richardson cone by the vector $(\De, \De, \De')$. It would be of interest to extend these results to the entire Littlewood-Richardson cone.

\section{Acknowledgements}
I thank Ketan Mulmuley for suggesting the question of computing Littlewood-Richardson coefficients, and Ravi Kannan for bringing \cite{KannanVempala} to my notice.

\end{document}